\documentclass[12pt, english, a4paper,reqno]{amsart}

\usepackage{graphicx}
\usepackage{tikz}
\usepackage{capt-of}
\usepackage[applemac]{inputenc}
\usepackage{hyperref}

\usepackage{amsmath,amssymb}
\usepackage{bm}
\usepackage{euscript,color}
\usepackage{babel}
\usepackage{amsthm}
\usepackage{amsfonts}
\usepackage{latexsym}
\usepackage{fancyhdr}
\usepackage{enumerate}

\usepackage[numbers]{natbib}

\usepackage[left=2.5cm, top=3cm, bottom=3cm, right=2.5cm]{geometry}
\usepackage{soul}
\newcommand{\R}{\mathbb{R}}

\newcommand{\Hom}{\mathop{\mathrm{Hom}}\nolimits}

\newcommand{\comment}[1]{}

\makeatletter
\renewcommand{\section}{\@startsection%
{section}
{1}
{0mm}
{1.5\bigskipamount}
{0.5\bigskipamount}
{\centering\normalsize\sc}}

\renewcommand{\paragraph}{\@startsection%
{paragraph}
{4}
{0mm}
{\bigskipamount}
{-1.25ex}
{\normalsize\sl}}

\def\provedboxcontents#1{$\square$}

\makeatother

\newtheoremstyle{thm}{}{}{\slshape}{}{\scshape}{.}{0.5em}{}

\newtheoremstyle{def}{}{}{}{}{\scshape}{.}{0.5em}{}

\newtheoremstyle{rmk}{}{}{}{}{\scshape}{.}{0.5em}{}

\newtheoremstyle{claim}{}{}{}{}{\slshape}{.}{0.5em}{}

\newtheorem{newstatement}{newstatement}

\newtheorem{theorem}[newstatement]{Theorem}

\newtheorem{proposition}[newstatement]{Proposition}

\newtheorem*{conjecture*}{Conjecture}
\newtheorem*{prop*}{Proposition}
\newtheorem{definition}{Definition}

\theoremstyle{def}

\theoremstyle{rmk}
\newtheorem{remark}[newstatement]{Remark}
\newtheorem{example}[newstatement]{Example}

\theoremstyle{claim}

\expandafter\let\expandafter\oldproof\csname\string\proof\endcsname
\let\oldendproof\endproof
\renewenvironment{proof}[1][\proofname]{%
  \oldproof[\slshape #1]%
}{\oldendproof}

\let\geq\geqslant
\let\leq\leqslant

\let\phi\varphi
\let\epsilon\varepsilon

\renewcommand{\emph}[1]{{\slshape #1}}
\renewcommand{\em}{\sl}

\title[Robustness of statistical models]{Robustness of statistical models}

\author{Andrea Loi}
\address{Andrea Loi, Dipartimento di Matematica e Informatica \\
         Universit\`a di Cagliari, Italy.}
         \email{loi@unica.it}

\author{Stefano Matta}
\address{Stefano Matta, Dipartimento di Scienze economiche e Aziendali \\
         Universit\`a di Cagliari, Italy.}
         \email{smatta@unica.it}

\date{}

\thanks{The first author was supported by INdAM. GNSAGA - Gruppo Nazionale per le Strutture Algebriche, Geometriche e le loro Applicazioni. Both authors were supported by STAGE - Funded by Fondazione di Sardegna.}

\begin{document}
\begin{abstract}
A  statistical structure $(g, T)$ on a smooth manifold $M$ induced by 
 $(\tilde M, \tilde g, \tilde T)$
is said to be {\em robust} if there exists an open neighborhood of $(g,T)$ in the fine $C^{\infty}$-topology consisting
of  statistical structures induced by  $(\tilde M, \tilde g, \tilde T)$.
Using Nash--Gromov implicit function theorem,
we show robustness of
the generic statistical structure  induced on $M$ by the standard  linear statistical structure on  ${\R}^N$, for $N$ sufficiently large.


{\it{Subj.Class}}: 53B12, 53C05, 53C42, 58C15. 

\end{abstract}
\maketitle

\vspace{0.3in}

\section{Introduction}\label{introduction}

The concept of statistical manifold \cite{lau}
provides an intrinsic approach and 
a useful abstraction to encompass various concepts and results 
in information geometry.
A statistical manifold is a manifold $M$
endowed with a statistical structure $(g,T)$,
where $g$ is a Riemannian metric and $T$ is a 3-symmetric tensor, 
which generalize the Fisher metric and 
the Amari-Chentsov tensor, respectively \cite{lau}.

Recently, \cite{hong} has positively addressed a 
question raised by \cite{lau} on whether a statistical manifold $(M,g,T)$
is a statistical model, i.e.
a smoothly parametrized family of probability measures on some sample space 
$\Omega$, $\mathcal P(\Omega)$, whose parameters belong to $M$.
The answer has been provided by \cite{hong}
showing the existence of an immersion of any statistical manifold in some $\mathcal P(\Omega)$, which preserves the statistical structure.
More precisely (see \cite{hong,jostbook}),
any statistical manifold admits an isostatistical embedding in 
$P(\Omega)$ endowed with the statistical structure represented by the 
Fisher metric and the Amari-Chentsov tensor.

We recall that an immersion $h:(M,g,T)\to (\tilde M, \tilde g, \tilde T)$
is isostatistical if it preserves the statistical structure,
i.e. $f^*\tilde g=g$ and $f^*\tilde T=T$. The statistical structure $(g, T)$ on
$M$ is then said to be {\em statistically induced} by  $(\tilde M, \tilde g, \tilde T)$. 
Hence it follows from this definition that a
probability density for the structure $(\tilde g, \tilde T)$, 
$p:\Omega\times\tilde M\to \R$, induces a probability density
for $(g,T)$. Observe that, as highlighted by \cite{jostbook},
this immersion, being  metric and tensor preserving,
can be seen as an \lq\lq intrinsic counterpart''
of sufficient statistic.

In this paper we follow this intrinsic approach.
Our aim is to study the robustness property  of
the  class of statistical structures $\{(g,T)\}$ on a manifold $M$,
which are statistically  induced by $(\tilde g, \tilde T)$,
the statistical structure of a manifold $\tilde M$.
We provide the following definition of robustness.

\begin{definition}
A  statistical structure $(g, T)$ on a smooth manifold $M$ induced by 
 $(\tilde M, \tilde g, \tilde T)$
is said to be {\em robust} if there exists an open neighborhood of $(g,T)$ in the fine $C^{\infty}$-topology consisting
of  statistical structures induced by  $(\tilde M, \tilde g, \tilde T)$.
\end{definition}

We think that this investigation is either natural,
since the function space  $\{(g, T)\}$
can be equipped with the 
fine (Whitney) $C^{\infty}$-topology, which coincides with the ordinary $C^{\infty}$-topology if $M$ is compact, either (hopefully) interesting.

The tool used in our analysis, Nash-Gromov implicit function theorem \cite{na,gr,g-r},
highlights a case of special interest, i.e. 
when $(\tilde M, \tilde g, \tilde T)$ 
is the standard linear statistical manifold, namely $\tilde M=\R^N$, 
$\tilde g=g_{can}=\sum_i^Ndx_i^2$ and $\tilde T=T_{can}= \sum_i^Ndx_i^3$. Hence,
in the sequel, by a {\em $N$-induced statistical structure} we will mean a statistical structure on $M$ induced by  $({\R}^N, g_{can}, T_{can})$. 

The main result  of the present paper is the following  theorem,  which shows that,
for $N$-induced statistical structures, robustness generically holds true,
if $N$ is assumed to be  sufficiently large. In other words, 
the space of robust $N$-induced statistical structures
is dense in the space of $N$-induced statistical structures.

\begin{theorem}\label{mainteor} 
Let $M$ be a  smooth  $n$-dimensional manifold
and let $(g_0, T_0)$ be a $N$-induced statistical structure on $M$.
Assume $N\geq \frac{n(n^2+9n+20)}{6}$.
Then $(g_0, T_0)$ can be approximated in the
Whitney $C^{\infty}$-topology by robust
$N$-induced statistical structures.

\end{theorem}


The reader may notice that our result
can be interpreted as a (local) variant
of the celebrated Nash's isometric immersion theorem
\cite{na}, which says that 
every $n$-dimensional Riemannian manifold $M^n$ can be isometrically
embedded in some ${\R}^{N(n)}$ 
endowed with the flat metric.
Indeed, the above statement is weaker 
than Nash's because there is an 
obstruction stemming from the invariance
of the norm of the $3$-symmetric tensor
which prevents a straightforward generalization of Nash's theorem.
For example, $(S(2)^n_+,g_+,T_+)$, the n-dimensional positive upper 
sphere of radius $2$
endowed with the metric $g_{can_{|S(2)^n_+}}$ and the tensor 
$\sum_{i=1}^n\frac{dx_i^3}{x_i}_{|S(2)^n_+}$,
which identifies the space of all positive probability measures
on a sample space of $n+1$ elementary events
endowed with the Fisher metric and the Amari-Chentsov tensor,
does not admit any isostatistical immersion on $(\R^N,g_{can},T_{can})$,
even if $T_{can}$ is multiplied by a positive constant
(the reader is referred to \cite[Sec 4.5.2]{jostbook} for 
 obstructions for the existence of an isostatistical immersion between 
 statistical manifolds).
The fact that $S(2)^n_+$ is not compact plays a crucial role.
In fact, \cite{hong} proves that any $n$-dimensional  {\em compact} statistical manifold 
$(M,  g, T)$, can be isostatistically embedded
into $({\R}^N, g_{can}, aT_{can})$, for a suitable $a>0$ and  a sufficiently large $N$.

In this paper we are not assuming any topological assumption on $M$ and, moreover, we are dealing with the standard $3$-symmetric tensor $T_{can}$ and not with its multiples.

The proof of Theorem \ref{mainteor} is based on Nash's implicit
function theorem for {\em infinitesimally invertible}
differential operators. Roughly speaking,  the idea of the proof of Theorem \ref{mainteor} is as follows. Since $(g_0, T_0)$ is $N$-induced, then there exists a smooth immersion 
$f_0:M\rightarrow {\R}^N$ such that $f_0^*g_{can}=g_0$ and $f_0^*T_{can}=T_0$.
The strategy is to show that the linearization $L_{can}$
of the smooth operator
${\mathcal D}_{can}$, which assigns   
to each smooth immersion $f:M\rightarrow {\R}^N$
the induced statistical structure
$(g, T)=(f^*g_{can}, f^*T_{can})$ on $M$,
can be infinitesimally inverted.

This paper is organized as follows.
In Section \ref{sectlinear}
we derive the linearization formula for the differential
operator ${\mathcal D}_{can}$ which corresponds
to the statistical structures 
under study. In Section \ref{proofmainteor},
after introducing and  discussing  the notion of free statistical  maps, which is relevant
to define the class of maps where Gromov's tecnique is applicable, namely where the linearization is invertible, we prove Theorem \ref{mainteor}.

\section{The operator ${\mathcal D}_{can}$ and its linearization $L_{can}$}\label{sectlinear}


Our study of statistically inducing  maps
follows the same approach and 
uses the same terminology as in \cite{gr}, where the reader is referred 
to for a general discussion on induced geometric structures
developed in the context of Nash's
immersion theory.
The key tool is a Nash-type implicit function
theorem proved by Gromov for a  special class
of differential operators
(see Section $2.3.1$ in \cite{gr}
for its various formulations and refinements).
A general criterion for the validity of the Nash-Gromov implicit function
theorem
is the {\em infinitesimal 
invertibility} of the relevant differential operator
and, in fact, we will 
work it out explicitly
for the inducing  differential operator ${\mathcal D}_{can}$,
namely  the operator
which assigns   
to each smooth immersion
the induced statistical structure
$(g, T)=(f^*g_{can}, f^*T_{can})$ on $M$ for the fixed
pair $(g_{can} ,T_{can})$ on ${\R}^N$. 
More precisely, the operator 
${\mathcal D}_{can}:\{f\}\rightarrow \{(g, T)\}$
is a differential operator between the space of  smooth immersions $M\rightarrow {\R}^N$
and the space of statistical structures on $M$ (both spaces equipped with the fine $C^{\infty}$-topology).

Observe that
Riemannian metrics $g$ (resp. symmetric $3$-tensors $T$) on 
$M$ are viewed as  smooth sections
$g:M\rightarrow S^2(M)$ (resp. $T:M\rightarrow S^3(M)$)
where $S^2(M)$ (resp.  $S^3(M)$) denotes  the symmetric
square (resp. the symmetric cube) of the cotangent bundle of $M$.
This allows us to interpret our pair of structures
(metric, $3$-tensor)
$=(g, T)$
as sections  
$M\rightarrow S^2(M)\oplus S^3(M)$.

\vspace{0.3cm}
\noindent
{\bf The linearization of the operator ${\mathcal D}_{can}$}

\noindent
Here we construct the linearization
of the operator ${\mathcal D}_{can}$. 
In easy terms, this linearization, denoted by $L_{can}$,
is the differential
of ${\mathcal D}_{can}$ at $f\in \{f\}$
and so it is a linear operator from the tangent space
to the space $\{f\}$
of smooth immersions $M\rightarrow {\R}^N$, say $T_{f}\{f\}$, 
to $T_{(g,T)}\{(g, T)\}$.
 Observe that, due to the above splitting
$S^2(M)\oplus S^3(M)$,
one can decompose the operator 
${\mathcal D}_{can}$ 
into the sum of two operators,
$${\mathcal D}_{can}={\mathcal D}_{g_{can}}\oplus {\mathcal D}_{T_{can}}:\{f\}
\rightarrow \{(g, T)\},$$
where, for a given smooth immersion
$f:M \rightarrow {\R}^N$,
$${\mathcal D}_{g_{can}}(f):=
f^*g_{can}=g$$
and
$${\mathcal D}_{T_{can}}(f):=f^*T_{can}=T.$$

We start by analyzing the linearization of 
${\mathcal D}_{g_{can}}$ and ${\mathcal D}_{T_{can}}$. Although, by the previous decomposition, we can analyze 
the linearization of these two components separately, in the following,
for the resolution of the system (\ref{system1})+(\ref{system2}), we should consider them jointly as they depend on the same argument $f$.

\vspace{0.3cm}
\noindent
{\bf The linearization of the operator ${\mathcal D}_{g_{can}}$}

\noindent
Our first operator ${\mathcal D}_{g_{can}}$
in a neighborhood
of $x\in M$ equipped with local coordinates
$x_1,\dots, x_n$, can be expressed by
$${\mathcal D}_{g_{can}}(f)=\{g_{ij}=g_{can}(f_i, f_j )\},
i,\ j =1,\dots, n,$$
where $f_i=df(\frac{\partial}{\partial x_i})$, 
$i=1,\dots ,n$, denote the images of the vector fields
$\frac{\partial}{\partial x_i}$ on $M$
under the differential of $f$ and where 
$g_{ij}$ are the components of the metric $g=f^*g_{can}$
in our local coordinates.

The linearization of the operator ${\mathcal D}_{g_{can}}$
at $f$
is the linear operator
$$L_{g_{can}}:C^{\infty}(M, {\R}^N)\rightarrow S^2(M),$$
assigning to 
each vector field $y$ on $\R^N$ along
$f(M)$ a quadratic form $g$
on $M$.
We take a smooth $1$-parametric
family of smooth maps $f_t:M\rightarrow {\R}^N$,
$t\in [0, 1]$, such that $f_0=f$
and 
$\frac{df_t}{dt}_{|_{t=0}}=y$
for a given 
$y: M\rightarrow {\R}^N$ and
set  $y_i=\frac{\partial y}{\partial x_i}$,
$i=1, \dots n$.
Then (compare either  \cite[2.3.1]{gr} or \cite{na})
the expression for 
$L_{g_{can}}(y)=\frac{d}{dt}{\mathcal D}_{g_{can}}(f_t)_{t=0}$
in local coordinates $x_1,\dots ,x_n$
is as follows:
\begin{equation}\label{Lh}
y\mapsto
g_{can}(f_i, y_j)+
g_{can}(f_j, y_i),\ i, j=1, \dots ,n.
\end{equation}

\vspace{0.3cm}
\noindent
{\bf The linearization of the operator ${\mathcal D}_{T_{can}}$}

\noindent
The second operator
${\mathcal D}_{T_{can}}$
reads, in local coordinates, $x_1,\dots ,x_n$,  as 
$${\mathcal D}_{T_{can}}(f)=\{T_{ijk}=T_{can}(f_i, f_j, f_k  )\},
i,j, k=1,\dots, n,$$
where 
$T_{ijk}$ are the components of the $3$-symmetric tensor $T=f^*T_{can}$
in our local coordinates.
The linearization of the operator ${\mathcal D}_{T_{can}}$
at $f$
is the linear operator
$$L_{T_{can}}: C^{\infty}(M, {\R}^N)\rightarrow S^3(M),$$
As before we take a smooth $1$-parametric
family of maps $f_t:M\rightarrow {\R}^N$,
$t\in [0, 1]$ such that $f_0=f$
and 
$\frac{df_t}{dt}_{|_{t=0}}=y$
for a given 
$y: M\rightarrow {\R}^N$. 
Then (cf. \cite[3.1.4]{gr})
$L_{T_{can}}(y)=\frac{d}{dt}{\mathcal D}_{T_{can}}(f_t)_{t=0}$
is given by:
\begin{equation}\label{Lh2}
y\mapsto
T_{can}(y_i, f_j, f_k)+
T_{can}(f_i, y_j, f_k)+T_{can}(f_i, f_j, y_k) ,\ i, j,k=1, \dots ,n.
\end{equation}

\vspace{0.3cm}
\noindent
{\bf The inversion of the operator $L_{can}=(L_{g_{can}}, L_{T_{can}})$}

\noindent
To  (locally) invert 
the operator ${\mathcal D}_{can}$,
we invert
its linearization $L_{can}=(L_{g_{can}},L_{T_{can}})$.
This amounts to solving the equation 
\begin{equation}\label{LH}
L_{can}(y)=
(L_{g_{can}}(y), L_{T_{can}}(y))=
(g', T')
\end{equation}
where the right-hand side
$(g', T')$
consists of an arbitrary 
quadratic $2$-tensor $g'$ on $M$
and an arbitrary $3$-tensor  $T'$ on $M$, respectively.
In view of
(\ref{Lh}) and (\ref{Lh2}),
we express (\ref{LH})
by the following system
of p.d.e. in the unknowns
$y$:
\begin{equation}\label{system1}
g_{can}(f_i, y_j)+
g_{can}(f_j, y_i)=g_{ij}'
\end{equation}
\begin{equation}\label{system2}
T_{can}(y_i, f_j, f_k)+
T_{can}(f_i, y_j, f_k)+T_{can}(f_i, f_j, y_k)=T'_{ijk},
\end{equation}
where 
$g_{ij}'$ 
and $T'_{ijk}$,
$i,j, k=1,\dots n$,
are smooth
functions on $M$ 
representing,
in the local coordinates $x_i$, the components of 
$g'$ and $T'$, respectively.

Next, 
we impose 
two additional conditions 
for the field
$y$ (see \cite{da1} and \cite{na}),
namely 
\begin{equation}\label{h0}
g_{can}(f_i, y)=0, \ i=1, \dots, n,
\end{equation}
and 
\begin{equation}\label{h1}
T_{can}(f_j, f_k, y)=0, \ j, k=1, \dots, n,
\end{equation}
Now, we differentiate 
(\ref{h0}) 
and alternate the index $i$
and $j$. Hence
the system 
(\ref{system1}) together with the extra-condition
(\ref{h0})
becomes equivalent to:
\begin{equation}\label{Lhsimpl}
g_{can}(f_{ij}, y)
=-\frac{1}{2}g_{ij}',\  g_{can}(f_i, y)=0,  \ i, j=1, \dots , n
\end{equation}
where $f_{ij}=\partial_i\partial_jf$.
On the other hand, if we differentiate 
(\ref{h1}), we get
$$T_{can}(f_j, f_k,y_i)=
-T_{can}(f_{ij}, f_k, y)-T_{can}(f_j, f_{ik}, y), \ i, j, k=1, \dots , n.$$
Therefore the system (\ref{system2}) 
with the conditions
(\ref{h1}) is equivalent to 
$$T_{can}(f_i, f_{jk}, y)+T_{can}(f_j, f_{ik}, y)+T_{can}(f_k, f_{ij}, y)=$$
\begin{equation}\label{Lhsimpl2}
=  -\frac{1}{2}T'_{ijk},\ T_{can}(f_j, f_k, y)=0,  \ i, j, k=1, \dots , n
\end{equation}

Notice now that since $T_{can}=\sum_{i=1}^ndx_i^3$ and $g_{can}=\sum_{i=1}^ndx_i^2$,
one gets
\begin{equation}\label{eqlinkTg}
T_{can}(u, v, w)=g_{can}(u\odot v, w), \ \forall u, v, w\in {\R}^N,
\end{equation}
where, for for $u=(u_1, \dots, u_N)$ and  $v=(v_1, \dots, v_N)$, 
\begin{equation}\label{dot}
u\odot v:=(u_1v_1, \dots , u_Nv_N). 
\end{equation}

Therefore, by (\ref{Lhsimpl}) and (\ref{Lhsimpl2}), 
the issue of infinitesimally inverting  the operator
${\mathcal D}_{can}$ is 
reduced to find
the solution $y$
of the following  system (\ref{systfinal1})+(\ref{systfinal2}):
\begin{equation}\label{systfinal1}
g_{can}(f_{ij}, y)
=\hat{g}_{ij},\  g_{can}(f_i, y)=0,\  i\leq j,
\end{equation}
\begin{equation}\label{systfinal2}
g_{can}(f_i\odot f_{jk}+f_j\odot f_{ik}+f_k\odot f_{ij}, y)
=\hat{T}_{ijk},\ g_{can}(f_j\odot f_k, y)=0, \ \  i\leq j\leq k,
\end{equation}

in the unknown field $y$,
where $\hat{g}_{ij}:M\rightarrow\R$
and $\hat{T}_{ijk}:M\rightarrow \R$
are smooth functions.
For each $x\in M$, this system is an  {\em algebraic} system consisting of
\begin{equation}\label{mn}
m_n:=n+2s_n+\binom{n+2}{3}= {\frac{n(n^2+9n+14)}{6}}
\end{equation}\label{vsn}
equations, where 
$s_n:=\frac{n(n+1)}{2}$. 
 Notice that every solution of the system 
(\ref{systfinal1})+(\ref{systfinal2}) also  gives  a solution
of the original linearized system
(\ref{system1})+(\ref{system2}) with the extra conditions
(\ref{h0}) and (\ref{h1}).

\section{Free statistical maps and the proof of Theorem \ref{mainteor}}
\label{proofmainteor}

The previous discussion enables us to see how the   linearization of the operator $L_{can}$, 
expressed by the system (\ref{systfinal1})+(\ref{systfinal2})
(and the consequent infinitesimal invertibility of the differential
operator ${\mathcal D}_{can}$) can be used for obtaining our desired result 
(Theorem \ref{mainteor}). The key step 
 is to show  that the operator
${\mathcal D}_{can}$, which associates to each immersion
$f:M\rightarrow {\R}^N$ the induced statistical structure $(f^*g_{can}, f^*T_{can})$,
is an open map on a dense subset in the space of maps.
We call these maps,
which satisfy a certain regularity condition,
{\em free statistical maps} (see Definition \ref{free} below).
Our proof follows the line of reasoning of Theorem 0.4.A 
in \cite{da2} and Theorem 1.1  in \cite{daloi}.
In fact, both papers follow the same pattern
of the case of Riemannian isometric  
immersions (see \cite{na} and  also \cite{gr}), where the relevant 
regularity condition is {\em freedom} of the involved
map $f:M\rightarrow {\R}^N$, i.e. linear independence of the 
$n+\frac{n(n+1)}{2}$ vectors of the first and second partial
derivatives of $f$ (see Remarks \ref{freeN} and \ref{freeG} below).

\begin{definition}[Free statistical maps]\label{free}
\noindent Let $f:M\rightarrow {\R}^N$
be a smooth map and  fix 
local coordinates 
$x_1,\dots , x_n$ around  a point $x\in M$
and denote by $f_i$ and 
$f_{ij}$, $i,j=1, \dots, n$
the first and second derivatives of the map $f$ with respect to these coordinates.
The map  $f:M\rightarrow {\R}^N$
is called a {\em free statistical map}
if, for all $x\in M$,
the $m_n$ (see \eqref{mn}) vectors
\begin{equation}\label{vecfree}
\{f_i(x), \ f_{ij}(x),f_j(x)\odot f_k(x),f_i(x)\odot f_{jk}(x)+ f_j(x)\odot f_{ik}(x)+f_k(x)\odot f_{ij}(x)\}
\end{equation}
are linear independent, for every $x\in M$ and and for all $i\leq j\leq k$.

\end{definition}

\begin{remark}
It is not hard to see that Definition \ref{free}
does not depend on the choice of local coordinates.
\end{remark}

\begin{remark}\label{freeN}
Let $f:M\rightarrow \R^N$ be a smooth map. 
Denote by 
$T^1_f(x)\subset T^2_f(x)
\subset {\R}^N$
the first and second 
osculating space
respectively of the map $f$
at the given point $x\in M$.
Namely, 
$T^1_f(x)=df_{x}(T_{x}M)$
and $T^2_f(x)\subset {\R}^N$
is the subspace spanned
by $f_i(x)$ and 
$f_{ij}(x)$, $i,j=1, \dots, n$, at $x$.
Then the dimension of $T^2_f(x)$
can vary between 
$0$ and $\min (N, n+s_n)$,
for $s_n=\frac{n(n+1)}{2}$
and the map $f$ is \em{free} in the sense of Nash
if $\dim T^1_f(n)=n=\dim M$, 
$\dim T^2_f(n)=\frac{n(n+3)}{2}=n+s_n$ or, equivalently, 
the $n+s_n$ vectors
$\{f_i(x), \ f_{ij}(x)\}$
are linear independent, for every $x\in M$ and for all $i\leq j$.
\end{remark}

\begin{remark}\label{freeG}
In Gromov's terminology (see \cite[3.1.4]{gr}),  a smooth map $f:M\rightarrow \R^N$ is called $T_{can}$-free if the $s_n+\binom{n+2}{3}$ vectors
$$\{f_i(x)\odot f_{jk}(x)+ f_j(x)\odot f_{ik}(x)+f_k(x)\odot f_{ij}(x),f_j(x)\odot f_k(x)\}$$
are linear independent, for every $x\in M.$ and for all $i\leq j\leq k$. Hence our definition of free statistical map 
extends both Nash's freedom and  Gromov's $T_{can}$-freedom conditions.
\end{remark}

\begin{example}
When $M=\R^n$ it is not hard to see
that the map $f:\R^n\to \R^{m_n}$ given by 
$$\left(x_1, 2x_1, \dots , x_n, 2x_n,  \{x_jx_k\}_{j\leq k}, \{x_j+x_k\}_{j<k},\{x_p+x_q^2\}_{p, q=1, \dots n},
\{ x_a+x_bx_c\}_{a<b<c}\right)$$
(where the strings are ordered in lexicographic order) is a  free statistical map.
\end{example}

In the following proposition we prove that the operator 
$$(d{\mathcal D}_{can})_f=L_{can}: T_f\{f\}\rightarrow \{(g, T)\}$$
is invertible if $f$ is free statistical.

\begin{proposition}\label{critinv}
Let $f:M\rightarrow {\R}^N$
be a free statistical map.
Then the linear operator
$L_{can}$ is invertible over all of $M$
by some differential operator
$M_{f}$, i.e.
$L_{can}\circ M_{f}=id$.
\end{proposition}
\begin{proof}
It follows from Section \ref{sectlinear}
that we need to find a solution
$y$ of the system 
of equations (\ref{systfinal1})+(\ref{systfinal2}).
Since the map $f:M\rightarrow {\R}^N$
is free statistical,
it follows  that the solution of
 (\ref{systfinal1})+(\ref{systfinal2})
forms an affine bundle
over $M$ of rank    $N- m_n$.
Now, every affine bundle
admits a section over $M$. To choose
it in a canonical way, one may use any 
fixed auxiliary Riemannian metric 
on ${\R}^N$ (e.g., we can use $g_{can}$)
and then take as canonical solution,
say $y_0$, the solution
$y$ of (\ref{systfinal1})+(\ref{systfinal2})
which has the minimal norm with respect
to this metric at every point
$f(x)\in {\R}^N$
(see, e.g., \cite{gr}, \cite{g-r}, \cite{na}).
Finally, we define the 
inverse $M_{f}$ of $L_{can}$
by 
$M_{f}(g', T')=y_0.$
\end{proof}

To make sure that the results we get 
are non-empty, we show the following:

\begin{proposition}\label{mainprop}
For $N\geq \frac{n(n^2+9n+20)}{6}$,
generic maps
$f:M\rightarrow {\R}^N$ are
free statistical.
\end{proposition}
\begin{proof}
We shall interpret   {\em non} free statistical condition
as a singularity in the space $J^2(M, {\R}^N)$
of $2$-jets of our maps $M\rightarrow {\R}^N$, so that we can use
an argument based on
{\em Thom's transversality theorem}. 
Recall  that the $2$-jet, $J^2_f(x)$, of a given smooth  map $f:M\rightarrow {\R}^N$
at the point $x$ is given by:
$$J^2_f(x)=(x, f(x), Df_x, D^2f_x),$$
where  $Df_x:T_xM\rightarrow T_x\R^N=\R^n$ (resp. $D^2f_x:S^2(T_xM)\rightarrow T_x\R^N=\R^n$ ) is the first (resp. second) derivative of $f$ at $x$,
and where $S^2(T_xM)$ denotes the symmetric square
of $T_xM$.
For fixed $x\in M$ consider thet set
$$J^2_{x}=\{(x, y, \alpha, \beta)\ | \  y\in \R^N, \alpha\in \Hom(T_xM, T_y\R^N), \beta\in \Hom(S^2(T_xM), T_y\R^N) \}$$
and the $2$-jet bundle $J^2(M, \R^N)=\bigsqcup_{x\in M}J^2_x$. 
Then $J^2(M, \R^N)$ inherits the structure of  smooth bundle over $M$ with fibers 
$J^2_{x}$ and natural projection
$$J^2(M, \R^N)\rightarrow M, (x, y, \alpha, \beta)\mapsto x.$$
Thus,  using  the $2$-jets   of  a smooth function $f:M\rightarrow \R^N$
one can  construct  the  smooth section  of this bundle, namely the  smooth map
$$J^2_f: M\rightarrow J^2(M, \R^N), x\mapsto J^2_f(x).$$ 
If we fix local coordinates 
$x_1,\dots , x_n$ around $x\in M$, then 
the $2$-jet $J^2_f(x)$ of a given map $f:M\rightarrow {\R}^N$
at the point $x$ is given by the first and second derivatives
$$J^2_f(x)=(x, f(x), f_i(x),f_{ij}(x)),
\,\, i, j=1,\dots ,n.$$
We also notice that the {\em non} free statistical regularity
at $x\in M$ depends on $J^2_f(x)$ and hence we can define the
subspace $\Sigma_{x}\subset J^2_x$ consisting of 
$2$-jets of {\em non} free statistical maps.
Let $M(m_n, N)$ be the set of $m_n\times N$ matrices with real entries,
where $m_n$ is defined by \eqref{mn}.
 Then it follows by Definition \ref{free} that  
 $\Sigma _{x}$ can be identified with the matrixes  of  $M(m_n, N)$ 
of rank strictly  less than $m_n$. Thus (cf., e.g., 
\cite{avz}) $\Sigma _{x}\subset M(m_n, N)$
is a stratified manifold of codimension
$N-m_n+1$.
Therefore the set 
$\Sigma =\cup _{x\in M}\Sigma_x
\subset \mathop{\hbox{J}}^2(M, {\R}^N)$,
which fibers over $M$,
is a stratified manifold 
of codimension $N-m_n+1$.
Now, by the very definition
of $\Sigma$,
it follows that a map $f: M\rightarrow {\R}^N$
is free statistical
iff $J^2_f(M)\subset {\mathop{\hbox{J}}}^2(M, {\R}^N)$
does not meet $\Sigma$.
Finally, (the special case of)
Thom's transversality theorem
(see , e.g. \cite{gr} Corollary $D ^{'}$, p. $33$)
tells us that generic maps do have the property
$J^{2}_f(M)\cap \Sigma =\varnothing$
iff $N-m_n+1\geq n+1$
or equivalently $N\geq \frac{n(n^2+9n+20)}{6}$.
\end{proof}

\vspace{0.3cm}

\noindent
{\bf Proof of Theorem \ref{mainteor}}
By assumption, $(g_0=f_{0}^*g_{can}, T_{0}=f_0^*T_{can})$ for 
a smooth map $f_0:M\rightarrow {\R}^N$, $N\geq \frac{n(n^2+9n+20)}{6}$.
Then, by Proposition \ref{mainprop}, there exists a free statistical map,  say $f_1:M\rightarrow {\R}^N$, 
which is arbitrarily  $C^{\infty}$-close  to the map $f_0$.
It follows that  the induced  statistical structure  ${\mathcal D}_{can}(f_1)=(g_1=f_1^*g_{can}, T_1=f_1^*T_{can})$ is  
 $C^{\infty}$-close to $(g_0, T_0)$. It remains to prove that $(g_1,T_1)$ is robust.
We know
by Proposition \ref{critinv}
that the linearization of the operator
${\mathcal D}_{can}$ at $f_1$
admits an inverse (or,
using the terminology in \cite{gr},
that
the operator 
${\mathcal D}_{can}$ is infinitesimally 
invertible at $f_1$).
This allows us to apply the Nash-Gromov's implicit function
theorem to deduce that ${\mathcal D}_{can}$ is an open
operator from a neighborhood of $f_1$
to a neighborhood ${\mathcal U}$ of 
${\mathcal D}_{can}(f_1)$.
Therefore all the statistical structures $(g, T)$ 
in ${\mathcal U}$ are $N$-induced and this concludes the proof of Theorem \ref{mainteor}.

\end{document}